\documentclass[12pt]{article}
 
 \usepackage{amsmath,amssymb,amscd,amsthm,esint}
 
\usepackage{graphics,amsmath,amssymb,amsthm,mathrsfs}

\usepackage{graphics,amsmath,amssymb,amsthm,mathrsfs,amsfonts}

\usepackage{sidecap}
\usepackage{float}
\usepackage{extarrows}
\usepackage{booktabs}
\usepackage{verbatim}
\usepackage{hyperref}
\usepackage[usenames,dvipsnames]{xcolor}

\newtheorem{thm}{Theorem}[section]
\newtheorem{lemma}[thm]{Lemma}

\theoremstyle{definition}
\newtheorem{remark}[thm]{Remark}

\def\XXint#1#2#3{{\setbox0=\hbox{$#1{#2#3}{\int}$}
         \vcenter{\hbox{$#2#3$}}\kern-.5\wd0}}

\def\R{\mathbb{R}}
\def\C{\mathbb{C}}
\def\e{\varepsilon}

\numberwithin{equation}{section}

\begin{document}

\title{Extrapolation for the $L^p$ Dirichlet Problem \\ in Lipschitz domains}

\author{Zhongwei Shen\thanks{Supported in part by NSF grant DMS-1600520.}}
\date{}
\maketitle

\begin{abstract}

Let $\mathcal{L}$ be a second-order linear elliptic operator with complex  coefficients.
We show that if the $L^p$ Dirichlet problem for the elliptic system $\mathcal{L}(u)=0$ in a fixed Lipschitz domain
$\Omega$ in $\R^d$ is solvable   for some $1<p=p_0< \frac{2(d-1)}{d-2}$, then it is solvable for all $p$ satisfying
$$
p_0<p< \frac{2(d-1)}{d-2} +\e.
$$
The proof is based on a real-variable argument.
It only requires that local solutions of $\mathcal{L}(u)=0$ satisfy a boundary Cacciopoli inequality. 

\end{abstract}

\section{\bf Introduction}\label{section-1}

In this paper we consider the $L^p$ Dirichlet problem for an $m\times m$ second-order elliptic system,
\begin{equation}\label{DP}
\left\{
\aligned
\mathcal{L} (u) &=0& \quad & \text{ in } \Omega, \\
u& = f\in L^p (\partial\Omega; \mathbb{C}^m) &\quad & \text{ on } \partial\Omega,\\
N(u) & \in L^p(\partial\Omega),
\endaligned
\right.
\end{equation}
where $\Omega$ is a bounded Lipschitz domain in $\R^d$ and $N(u)$ denotes the (modified) nontangential 
maximal function of $u$.
The operator  $\mathcal{L}$ in (\ref{DP})  is a second-order linear elliptic operator with complex  coefficients.
It may contain  lower oder terms and needs not to be in divergence form.
Instead we shall impose  the following condition.
\begin{quote}

Let $r_0=\text{diam}(\Omega)$.
There exists constants $\kappa>0$ and $c_0>0$ such that the boundary Cacciopoli inequality 
\begin{equation}\label{BC}
\int_{B(x_0, r)\cap \Omega } |\nabla u|^2\, dx \le \frac{\kappa}{r^2} \int_{B(x_0, 2r)\cap\Omega} |u|^2\, dx
\end{equation}
holds,
whenever $x_0\in \partial\Omega$, $0<r<c_0 r_0$, and
  $u\in W^{1, 2}(B(x_0, 2r)\cap\Omega; \C^m)$ is a weak solution to $\mathcal{L}(u)=0$ in $B(x_0, 2r)\cap\Omega$ with
  $u=0$ on $B(x_0, 2r)\cap\partial\Omega$.
  
\end{quote}

\begin{thm}\label{main-theorem}
Let $\Omega$ be a fixed bounded Lipschitz domain in $\R^d$ and 
$
1<p_0< \frac{2(d-1)}{d-2}.
$
Let $\mathcal{L}$ be a second-oder linear elliptic  operator satisfying the condition (\ref{BC}).
Assume that for any $f\in C_0^\infty(\R^d; \C^m)$, there exists a weak solution $u\in W^{1, 2}(\Omega; \C^m)$
to $\mathcal{L}(u)=0$ in $\Omega$ such that
$u=f$ on $\partial\Omega$ in the sense of trace, and
$\| N(u)\|_{L^{p_0} (\partial\Omega)} \le C_0 \| f\|_{L^{p_0} (\partial\Omega)}$.
Then the weak solution $u$ satisfies the $L^p$ estimate
\begin{equation}\label{estimate-1.0}
\| N(u)\|_{L^{p} (\partial\Omega)} \le C \| f\|_{L^{p} (\partial\Omega)}
\end{equation}
for any $p$ satisfying
\begin{equation}\label{range-p}
p_0<p<\frac{2(d-1)}{d-2} +\e,
\end{equation}
where $\e>0$ depends only on $d$, $m$, $p_0$, $\kappa$, 
 $c_0$, $C_0$ and the Lipschitz character of $\Omega$.
The constant $C$ in (\ref{estimate-1.0}) depends on $d$, $m$, $p_0$,  $p$,
$\kappa$, $c_0$, $C_0$ and the Lipschitz character of $\Omega$.
\end{thm}

We remark that in the scalar case $m=1$ with real coefficients, the maximum principle
$\| u\|_{L^\infty(\Omega)} \le \| u\|_{L^\infty(\partial\Omega)}$ holds for weak solutions of
$\mathcal{L}(u)=0$ in $\Omega$.
It follows by interpolation that if the estimate (\ref{estimate-1.0}) holds for $p=p_0$,
then it holds for any $p_0<p\le \infty$.
However, it is known that the maximum principle or its weak version $\| u\|_{L^\infty(\Omega)}
\le C \| u\|_{L^\infty(\partial\Omega)}$ is not available  in Lipschitz domains for elliptic systems
or scalar elliptic equations with complex coefficients.
Theorem \ref{main-theorem} provides a partial solution to this problem.

The analogous of Theorem \ref{main-theorem} also holds if $\Omega$ is the region above a 
Lipschitz graph,
\begin{equation}\label{g-domain}
\Omega=\big\{ (x^\prime, x_d)\in \R^d: \ x_d >\psi (x^\prime) \big\},
\end{equation}
where $\psi: \R^{d-1} \to \R$ is a Lipschitz function with $\|\nabla \psi\|_\infty\le M$.

\begin{thm}\label{main-theorem-1}
Let $\Omega$ be a fixed graph domain in $\R^d$, given by (\ref{g-domain}),  and 
$
1<p_0< \frac{2(d-1)}{d-2}.
$
Let $\mathcal{L}$ be a second-oder linear elliptic  operator satisfying the condition (\ref{BC}) with $r_0=\infty$.
Assume that for any $f\in C_0^\infty(\R^d; \C^m)$, there exists a weak solution $u\in W_{loc}^{1, 2}(\overline{\Omega}; \C^m)$
to $\mathcal{L}(u)=0$ in $\Omega$ such that
$u=f$ on $\partial\Omega$ in the sense of trace, and
$\| N(u)\|_{L^{p_0} (\partial\Omega)} \le C_0 \| f\|_{L^{p_0} (\partial\Omega)}$.
Then the weak solution $u$ satisfies the estimate (\ref{estimate-1.0})
for any $p$ satisfying (\ref{range-p}),
where $\e>0$ depends only on $d$, $m$, $p_0$, $\kappa$,  $C_0$ and $M$.
The constant $C$ in (\ref{estimate-1.0}) depends on $d$, $m$, $p_0$,  $p$,
$\kappa$, $C_0$ and $M$.
\end{thm}

\begin{remark}
{\rm
Regarding the boundary Cacciopoli inequality (\ref{BC}) in a graph domain $\Omega$, 
consider the elliptic operator
\begin{equation}\label{operator}
\left(\mathcal{L}(u)\right)^\alpha
=-\frac{\partial}{\partial x_i} \left\{ a_{ij}^{\alpha\beta} (x) \frac{\partial u^\beta}{\partial x_j} \right\}
+ b_j^{\alpha\beta}(x) \frac{\partial u^\beta}{\partial x_j},
\end{equation}
where $1\le \alpha, \beta\le m$ and $1\le i, j\le d$
(the repeated indices are summed).
Assume that the coefficients $a_{ij}^{\alpha\beta} (x)$ are complex-valued bounded functions satisfying 
$\| a_{ij}^{\alpha\beta}\|_\infty\le \mu^{-1}$ and the
ellipticity condition 
\begin{equation}\label{ellipticity}
\text{\rm Re} \left( a_{ij}^{\alpha\beta} (x) \xi_j^\beta  \overline{\xi_i^\alpha}\right)
\ge \mu |\xi|^2
\end{equation}
for any $\xi =(\xi_i^\alpha)\in \mathbb{C}^{m\times d}$, where $\mu>0$.
Also assume that there exists some $\nu>0$ such that
\begin{equation}\label{B-cond}
|b_j^{\alpha\beta} (x)|\le \nu \delta (x)
\end{equation}
for any $x\in \Omega$, where $\delta (x)=\text{dist}(x, \partial\Omega)$.
Then there exists a constant $\nu_0>0$, depending only on $d$, $m$, $\mu$ and $M$, such that
if $\nu\le \nu_0$, the Cacciopoli inequality (\ref{BC}) holds for any $0<r<\infty$.
This may be proved by using Hardy's inequality.
In the case of a bounded Lipschitz domain, one only needs to assume 
(\ref{B-cond}) with $\nu\le \nu_0$ for $x$ sufficiently close to $\partial\Omega$ ($\delta(x)\le c_0 r_0$).
}
\end{remark}

\begin{remark}
{\rm
Let $d\ge 3$.
If the Dirichlet problem (\ref{DP}) is solvable for $p=p_0=\frac{2(d-1)}{d-2}$,
our argument gives the solvability for $p_0<p<p_0 +\e$.
}
\end{remark}

The $L^p$ boundary value problems for second-order elliptic equations and systems in Lipschitz domains
have been studied extensively. We refer the reader to
 \cite{Kenig-book, Shen-2006-ne, AA-2011, AAAHK-2011, HMM-2011, HKSJ-2015, HKMP-2015, DPR-2017} for references.
In particular,  the $L^2$ Dirichlet problem is solvable 
for elliptic systems with real constant coefficients satisfying the Legendre-Hadamard condition and the symmetry condition
\cite{FKV-1988, Dahlberg-Kenig-Verchota-1988, Fabes-1988, Gao-1991}.
It is also known that under the same assumption, 
the $L^p$ Dirichlet problem is solvable for $2-\e<p\le \infty$ if $d=3$ \cite{DK-1990}, 
and for $2-\e < p< \frac{2(d-1)}{d-3} +\e$ if $d\ge 4$ \cite{Shen-2006-ne}. 
More recent work in this area focuses on operators with complex coefficients or real coefficients without 
the symmetry condition \cite{AA-2011, AAAHK-2011, HMM-2011, HKSJ-2015, HKMP-2015}.

As in \cite{Shen-2006-ne}, the proof of Theorems \ref{main-theorem} and \ref{main-theorem-1} is based on a real-variable 
method, which may be regarded as a dual version of the celebrated Calder\'on-Zygmund Lemma.
The method originated in \cite{CP-1998} and was further developed in \cite{Shen-2005, Shen-2006-ne, Shen-2007}.
It reduces the $L^p$ estimate (\ref{estimate-1.0}) to the reverse H\"older inequality,
\begin{equation}\label{RH-00}
\left(\fint_{B(x_0, r)\cap\partial\Omega} |N(u)|^{q}\, d\sigma \right)^{1/q}
\le C \left(\fint_{B(x_0, 2r)\cap\partial\Omega} |N(u)|^{p_0} \, d\sigma \right)^{1/p_0}
\end{equation}
for $q=\frac{2(d-1)}{d-2}$ (for any $2<q<\infty$ if $d=2$), where $x_0\in \partial\Omega$,
$u$ is a weak solution to $\mathcal{L}(u)=0$ in $\Omega$ with $u=0$ in $B(x_0, 3r)\cap\partial\Omega$.
To prove (\ref{RH-00}), we replace $N(u)$ by $N^r(u)$, a localized nontangential maximal function at height $r$
(see Section 2 for definition), and use the observation 
\begin{equation}\label{RH-01}
\int_{B(x_0, r)\cap\partial\Omega}
|N^r (u)|^q\, d\sigma \le C \int_{B(0, 2r)\cap\Omega} | u(y)|^q \delta (y)^{-1}\, dy.
\end{equation}
The right-hand side of (\ref{RH-01}) is then handled by using Sobolev inequality and Hardy's inequality,
\begin{equation}\label{RH-02}
\int_{B(x_0, 2r)\cap \Omega} \frac{|u(y)|^2}{\delta (y)^2} \, dy
\le C \int_{B(x_0, 2r)\cap\Omega} |\nabla u|^2\, dy.
\end{equation}
The exponent $q=\frac{2(d-1)}{d-2}$ aries in the use of Sobolev inequality 
\begin{equation}\label{RH-03}
\|u\|_{L^{2(q-1)} (B(x_0, 2r)\cap\Omega)} \le C \|\nabla u\|_{L^2(B(x_0, 2r)\cap\Omega)}.
\end{equation}
It may be worthy to point out that $q$ is also the exponent in the boundary Sobolev
 inequality $\|u\|_{L^q(\partial\Omega)}\le C \| u\|_{H^{1/2}(\partial\Omega)}$.


\section{\bf Reverse H\"older inequalities}

Throughout this section we assume that $\Omega$ is the region above a Lipschitz graph in $\R^d$, given by (\ref{g-domain})
with $\|\nabla \psi\|_\infty\le M$.
A nontangential approach region at $z\in \partial\Omega$ is given by
\begin{equation}\label{Gamma}
\Gamma_a (z) =\big\{ x\in \Omega: \  |x-z|< a\,  \delta(x) \big\},
\end{equation}
where  $\delta(x)=\text{dist}(x, \partial\Omega)$ and $a>1+2M$.
We also need a truncated version 
\begin{equation}\label{Gamma-h}
\Gamma^h_a (z) =\big\{ x\in \Omega:  \ |x-z|< a\,  \delta(x) \text{ and } \delta (x)< h \big\},
\end{equation}
where $h>0$.
For $u\in L^2_{loc}(\Omega)$, the modified nontangential maximal function of $u$
is defined by
\begin{equation}\label{N}
N_a (u)(z)=\sup \left\{ 
\left( \fint_{B(x, (1/4)\delta (x))} |u|^2\right)^{1/2}: \ 
x\in \Gamma_a (z) \right\}
\end{equation}
for each $z\in \partial\Omega$. Similarly, we introduce 
\begin{equation}\label{N-h}
N^h_a (u)(z)=\sup \left\{ 
\left( \fint_{B(x, (1/4)\delta (x))} |u|^2\right)^{1/2}: \ 
x\in \Gamma^h_a (z) \right\}.
\end{equation}
The definitions of 
$N_a(u)$ and $N_a^h (u)$ are  same
if $\Omega$ is a bounded Lipschitz domain.
We will drop the subscript $a$ if there is no confusion.

\begin{lemma}\label{lemma-2.1}
Let $2\le q <\infty$. Then 
\begin{equation}\label{2.1-0}
N_a^h (u) (z)\le 
C\left( \int_{\Gamma_{2a}^{2h} (z)} |u (y)|^q \delta (y)^{-d} \, dy\right)^{1/q}
\end{equation}
for any $z\in \partial\Omega$,
where $C$ depends only on $d$ and $q$.
\end{lemma}

\begin{proof}
Fix $x\in \Gamma_a^h (z)$.
Let $y\in B(x, (1/4)\delta (x))$. Note that 
$$
\delta(y)\le \delta (x) +|x-y|< (5/4)\delta (x).
$$
Since $\delta(x)\le \delta(y) +|x-y|< \delta (y) + (1/4)\delta (x)$,
we obtain $(3/4)\delta (x)< \delta (y)$.
It follows that
$$
\aligned
|y-z| & \le |x-z| +|x-y|
< (a +(1/4)) \delta (x)\\
&\le (4/3) (a+(1/4)) \delta (y)
\le 2a \delta (y),
\endaligned
$$
where we have used the fact $a>1$.
Also observe that $\delta (y)<(5/4)\delta (x)< (5/4) h$.
Thus we have proved that $B(x, (1/4)\delta(x))\subset \Gamma_{2a}^{2h} (z)$.
This, together with H\"older's inequality,  gives
$$
\aligned
\left(\fint_{B(x, (1/4)\delta (x))} |u(y)|^2\, dy\right)^{1/2}
&\le 
\left(\fint_{B(x, (1/4)\delta (x))} |u(y)|^q\, dy\right)^{1/q}\\
&\le C \left(\int_{\Gamma_{2a}^{2h} (z)} | u(y)|^q \delta (y)^{-d}\, dy\right)^{1/q},
\endaligned
$$
where $C$ depends only on $d$ and $q$.
The inequality (\ref{2.1-0}) now follows by definition.
\end{proof}

Assume that $\psi (0)=0$.
For $r>0$, define
\begin{equation}
\aligned
D_r & =\big\{ (x^\prime, x_d)\in \R^d: \ |x^\prime|<r \text{ and } \psi(x^\prime)< x_d< 2(M+1)r \big\},\\
\Delta_r & =\big\{ (x^\prime, \psi(x^\prime))\in \R^d: \ |x^\prime|<r \big\}.
\endaligned
\end{equation}

\begin{lemma}\label{Hardy}
Suppose that $u\in H^1(D_r)$ and $u=0$ on $\Delta_r$. Then
\begin{equation}\label{H-1}
\int_{D_r} \frac{|u(x)|^2}{\widetilde{\delta}(x)^2}\, dx
\le 4
\int_{D_r} |\nabla u|^2\, dx,
\end{equation}
where $\widetilde{\delta}(x)=|x_d-\psi (x^\prime)|$.
\end{lemma}

\begin{proof}
Using
$u(x^\prime, \psi(x^\prime))=0$ and Fubini's Theorem, we obtain 
$$
\aligned
\int_{D_r} \frac{|u(x)|^2}{\widetilde{\delta} (x)^2}\, dx
&= \int_{|x^\prime|<r} \int_{\psi(x^\prime)}^{2 (1+M) r}
\frac{|u(x^\prime, x_d)|^2}{|x_d -\psi(x^\prime)|^2} \, d x_d d x^\prime\\
& \le  4\int_{|x^\prime|<r} \int_{\psi(x^\prime)}^{2(1+M) r} 
\left|\frac{\partial u}{\partial x_d} \right|^2 \, dx_d dx^\prime \\
&\le 4 \int_{D_r} |\nabla u|^2\, dx,
\endaligned
$$
where we have used the Hardy inequality (see e.g. \cite[p.272]{Stein-book})for the first inequality. 
\end{proof}
 
 The following lemma is one of the main steps in our argument.
 
\begin{lemma}\label{RH-lemma}
Let $u\in H^1(B(0, 6kr)\cap\Omega; \mathbb{C}^m)$ be a weak solution to 
$\mathcal{L}(u)=0$ in $B(0, 6kr)\cap\Omega$ with $u=0$ on $B(0, 6kr)\cap\partial\Omega$
for some $0<r<\infty$, where $k=10a(M+2)$.
Assume that 
\begin{equation}\label{C-1}
\int_{B(0, kr)\cap\Omega} |\nabla u|^2\, dx
\le \frac{C_0}{r^2} \int_{B(0, 2kr)\cap\Omega} |u|^2\, dx.
\end{equation}
Then
\begin{equation}\label{RH}
\left(\fint_{\Delta_{r}} |N_a^r (u)|^q\, d\sigma \right)^{1/q}
\le C 
\left(\fint_{\Delta_{2kr}} |N_a^{4kr} (u)|^2\, d\sigma \right)^{1/2},
\end{equation}
where $q=\frac{2(d-1)}{d-2}$ for $d\ge 3$ and $2<q<\infty$ for $d=2$.
The constant  $C$ depends only on $d$, $m$, $M$, $C_0$, and $q$ (if $d=2$).
\end{lemma}

\begin{proof}
We give the proof for the case $d\ge3$.
With minor modification,
the same argument works for $d=2$.
 It follows from (\ref{2.1-0}) and Fubini's Theorem that
$$
\aligned
\int_{\Delta_r} |N_a^r (u)|^q\, d\sigma
& \le C\int_{\Delta_r} \int_{\Gamma_{2a}^{2r} (z)} |u(y)|^q \delta (y)^{-d}\, dy d\sigma (z)\\
&\le  C \int_E  |u(y)|^q \delta(y)^{-1}\, dy,
\endaligned
$$
where 
$$
E= \bigcup_{z\in \Delta_r} \Gamma_{2a}^{2r} (z).
$$
Note that if $y\in E$, then $y\in \Gamma_{2a}^{2r} (z)$ for some $z\in \Delta_r$. 
Hence,
$$
\aligned
|y| & \le |y-z| +|z|< 2a \delta (y) +(1+M)r\\
& \le (4a +1+M) r\le 5a r,
\endaligned
$$
where we have used the fact $a\ge 1+M$.
This shows that $|y^\prime|\le 5a r$ and
$|y_d|<5a r$.
As a result, we obtain $E\subset D_{5ar}$.
Thus,
\begin{equation}\label{2.10}
\aligned
\int_{\Delta_r} |N_a^r (u)|^q\, d\sigma
&\le C \int_{D_{5ar}} |u(y)|^q \delta (y)^{-1}\, dy\\
& \le C \left(\int_{D_{5ar}} |u|^{2(q-1)} \, dy\right)^{1/2}
\left(\int_{D_{5ar}} \frac{|u(y)|^2}{\delta(y)^2}\, dy \right)^{1/2},
\endaligned
\end{equation}
where we have used the Cauchy inequality for the last step.

To bound the right-hand side of (\ref{2.10}), we first note that
$$
\frac{1}{\sqrt{2} (M+1)} |x_d-\psi (x^\prime)|
\le \delta (x)\le |x_d-\psi (x^\prime)|.
$$
In view of Lemma \ref{Hardy} we obtain 
\begin{equation}\label{2.11}
\int_{D_{5ar}} \frac{|u(y)|^2}{\delta(y)^2}\, dy 
\le C \int_{D_{5ar}} |\nabla u(y)|^2\, dy,
\end{equation}
where $C$ depends only on $M$.
Recall that $q=\frac{2(d-1)}{d-2}$. Thus $2(q-1)=\frac{2d}{d-2}$.
Since $u=0$ on $\Delta_{5ar}$, we may apply the Sobolev inequality to obtain 
\begin{equation}\label{2.12}
\left(\int_{D_{5ar}} |u|^{2(q-1)}\, dy \right)^{1/(2(q-1))}
\le C \left(\int_{D_{5ar}} |\nabla u|^2\, dy\right)^{1/2}.
\end{equation}
This, together with (\ref{2.10}) and (\ref{2.11}), leads to
\begin{equation}\label{2.13}
\aligned
\int_{\Delta_r} |N_a^r (u)|^q \, d\sigma
& \le C \left(\int_{D_{5ar}} |\nabla u|^2\, dy \right)^{q/2}\\
&\le C \left(\int_{B(0, 10a(M+2) r)\cap\Omega} |\nabla u|^2\, dy \right)^{q/2},
\endaligned
\end{equation}
where we have used the observation $D_{5ar}\subset B(0, 10a(M+2)r)$ for the last step.
Hence,
\begin{equation}\label{2.14}
\aligned
\left(\fint_{\Delta_r} |N_a^r (u)|^q \, d\sigma\right)^{1/q}
& \le C r^{\frac{d}{2}-\frac{d-1}{q}}
 \left(\fint_{B(0, 10a(M+2) r)\cap\Omega} |\nabla u|^2\, dy \right)^{q/2}\\
 &\le C \left(\fint_{B(0, 20a(M+2) r)\cap\Omega} | u|^2\, dy \right)^{/2},
\endaligned
\end{equation}
where we have used  the assumption (\ref{C-1}) for the last step.

Finally, we note that if $|x-y|<(1/5)\delta (y)$, then $|x-y|< (1/4)\delta (x)$.
Thus, by Fubini's Theorem,
$$
\aligned
\fint_{B(0,R)\cap \Omega} |u|^2\, dx
 &\le C \fint_{B(0, R)\cap \Omega}  
\left(\fint_{B(x, (1/4)\delta(x))} |u|^2\, dy\right) dx\\
&\le C \fint_{\Delta_R} |N_a^{2R} (u)|^2\, d\sigma
\endaligned
$$
for any $R>0$. This, together with (\ref{2.14}), yields the reverse H\"older inequality (\ref{RH}).
\end{proof}

We are now ready to prove the main result of this section.

\begin{thm}\label{RH-theorem}
Let $u\in H^1(B(0, 9kR)\cap\Omega; \mathbb{C}^m)$ be a weak solution to 
$\mathcal{L}(u)=0$ in $B(0, 9kR)\cap\Omega$ with $u=0$ on $B(0, 9kR)\cap\partial\Omega$
for some $0<R<\infty$, where $k=10a(M+2)$.
Assume that 
\begin{equation}\label{C-2}
\int_{B(z, r)\cap\Omega} |\nabla u|^2\, dx
\le \frac{C_0}{r^2} \int_{B(z, 2r)\cap\Omega} |u|^2\, dx
\end{equation}
for any $0<r<3kR$ and any $z\in B(0,3kR)\cap\partial\Omega$.
Then for any $0<r<R$,
\begin{equation}\label{RH-2}
\left(\fint_{\Delta_{r}} |N_a^{4kR} (u)|^q\, d\sigma \right)^{1/q}
\le C 
\fint_{\Delta_{2r}} N_a^{4kR} (u)\, d\sigma ,
\end{equation}
where $q=\frac{2(d-1)}{d-2}$ for $d\ge 3$ and $2<q<\infty$ for $d=2$.
The constant  $C$ depends only on $d$, $m$, $M$, $C_0$, and $q$ (if $d=2$).
\end{thm}

\begin{proof}
We first show that for any $0<r<R$,
\begin{equation}\label{RH-3}
\left(\fint_{\Delta_{r}} |N_a^{4kR} (u)|^q\, d\sigma \right)^{1/q}
\le C 
\left(\fint_{\Delta_{2kr}} |N_a^{4kR} (u)|^2\, d\sigma\right)^{1/2}.
\end{equation}
Let $z\in \Delta_r$ and $x\in\Gamma_a^{4kR} (z)$.
If $\delta(y)<r$, we have
$$
\left(\fint_{B(x, (1/4)\delta(x))} |u|^2\right)^{1/2}
\le N_a^r (u)(z).
$$
Suppose $\delta(x)>r$. It follows by a simple geometric observation that
there exists a constant $c_0\in (0, 1)$, depending only on $d$, $M$ and $a$, such that
$$
|\big\{ y\in \Delta_{\Delta_{2kr}}: \ x\in \Gamma_a^{4kR} (y) \big\}|
\ge c_0 r^{d-1}.
$$
This implies that
$$
\left(\fint_{B(x, (1/4)\delta(x))} |u|^2\right)^{1/2}
\le C \fint_{\Delta_{2kr}} 
N_a^{4kR}(u)\, d\sigma.
$$
Hence, for any $z\in \Delta_r$, 
\begin{equation}\label{RH-3.01}
N_a^{4kR} (u) (z)
\le N_a^r (u) (z)+ C \fint_{\Delta_{2kr}} N_a^{4kR} (u)\, d\sigma,
\end{equation}
which, together with (\ref{RH}), gives (\ref{RH-3}).

The fact that (\ref{RH-3}) implies (\ref{RH-2}) follows from a convexity argument, found in \cite{FS-H-p}.
For $z=(z^\prime, \psi(z^\prime))\in \partial\Omega$ and $r>0$, define the surface ball $\Delta_r (z)$
on $\partial\Omega$ by
\begin{equation}\label{Q}
\Delta_r (z)
=\big\{ (x^\prime, \psi (x^\prime))\in \R^d: \
|x^\prime- z^\prime|< r \big\}.
\end{equation}
Note that $\Delta_r =\Delta_r(0)$.
By translation the inequality (\ref{RH-3}) continues to hold if $\Delta_r$ and
$\Delta_{2kr}$ are replaced by $\Delta_r (z)$ and $\Delta_{2kr}(z)$, respectively.
Let $0<s<t<1$.
We may cover $\Delta_{sr}$ by a finite number of surface balls $\{ \Delta_{c(t-s)r} (z_\ell)\}$
with the property $\Delta_{2kc(t-s)r} (z_\ell)\subset \Delta_{tr}$.
Note that
$$
\aligned
\fint_{\Delta_{sr}} |N_a^{4kR} (u)|^q\, d\sigma 
&\le Cs^{1-d} (t-s)^{d-1} \sum_\ell \fint_{\Delta_{c(t-s)r} (z_\ell)} |N_a^{4kR}(u)|^q\, d\sigma \\
&\le Cs^{1-d}(t-s)^{d-1} \sum_\ell \left(\fint_{\Delta_{2kc(t-s)r} (z_\ell)} |N_a^{4kR}(u)|^2\, d\sigma \right)^{q/2}\\
&\le C s^{1-d} (t-s)^{d-1}
\left( \sum_\ell \fint_{\Delta_{2kc(t-s)r}(z_\ell)} |N_a^{4kR} (u)|^2\, d\sigma  \right)^{q/2}\\
&\le C s^{1-d} (t-s)^{d-1}t^{\frac{q}{2} (d-1)} (t-s)^{-\frac{q}{2}(d-1)}
\left(\fint_{\Delta_{tr}} |N_a^{4kR} (u)|^2 \, d\sigma \right)^{q/2}.
\endaligned
$$
It follows that for any $0<s<t<1$,
\begin{equation}\label{RH-5}
\left(\fint_{\Delta_{sr}} |N_a^{4kR} (u)|^q\, d\sigma \right)^{1/q}
\le C s^{\frac{1-d}{q}} t^{\frac{d-1}{2}}
(t-s)^{(d-1)(\frac{1}{q}-\frac12)}
\left(\fint_{\Delta_{tr}} |N_a^{4kR} (u)|^2\, d\sigma \right)^{1/2}.
\end{equation}
Write $\frac{1}{2}=\frac{\theta}{q} +\frac{\theta}{1}$, where $\theta\in (0,1)$.
By H\"older's inequality,
\begin{equation}\label{RH-6}
\left(\fint_{\Delta_{tr}} |N_a^{4kR} (u)|^2\, d\sigma \right)^{1/2}
\le 
\left(\fint_{\Delta_{tr}} |N_a^{4kR} (u)|^q\, d\sigma \right)^{(1-\theta)/q}
\left(\fint_{\Delta_{tr}} |N_a^{4kR} (u)|\, d\sigma \right)^\theta.
\end{equation}
Let 
$$
I(t)=\left(\fint_{\Delta_{tr}} |N_a^{4kR} (u)|^q\, d\sigma \right)^{1/q}
/\fint_{\Delta_{r}} N_a^{4kR} (u)\, d\sigma.
$$
By (\ref{RH-5}) and (\ref{RH-6}) we obtain 
$$
I(s)\le C s^{\frac{1-d}{q}} t^{(d-1)(\frac12-\theta)}
(t-s)^{(d-1)(\frac{1}{q}-\frac12)}
\big[ I(t)\big]^{1-\theta}.
$$
Hence,
$$
\log I(s)  \le \log \left( C s^{\frac{1-d}{q}} t^{(d-1)(\frac12-\theta)}
(t-s)^{(d-1)(\frac{1}{q}-\frac12)}\right)
+ (1-\theta) \log I (t).
$$
Let $s=t^b$, where $b>1$ is chosen so that $b^{-1}>1-\theta$.
We integrate the inequality above in $t$ with respect to $t^{-1} dt$ over the interval
$(1/2, 1)$.
This gives
$$
\frac{1}{b} \int_{(1/2)^b}^1 \log I(t)\frac{dt}{t}
\le C + (1-\theta) \int_{1/2}^1 \log I(t)\, \frac{dt}{t}.
$$
It follows that
$$
(\frac{1}{b} -\theta) \int_{1/2}^1  \log I(t)\frac{dt}{t} \le C.
$$
Since $I(t)\ge c I(1/2)$ for $t\in (1/2, 1)$, we obtain $I(1/2)\le C$,
which gives (\ref{RH-2}).
\end{proof}

\begin{remark}\label{self}
{\rm
By translation  the inequality (\ref{RH-2}) continues to hold if $\Delta_r$ and
$\Delta_{2r}$ are replaced by surface balls $\Delta_{r}(z)$ and $\Delta_{2r} (z)$, respectively,
where $z\in \Delta_R$.
In the case $d\ge 3$, (\ref{RH-2}) in fact holds for some
$\overline{q}=\frac{2(d-1)}{d-2} +\e$, where $\e>0$ depends only on 
$d$, $m$, $M$ and $C_0$.
This follows from the well known self-improving property of the reverse H\"older
inequality. 
}
\end{remark}


\section{\bf Proof of Theorem \ref{main-theorem-1}}

Throughout this section  we assume that $\Omega$ is  a graph domain, given by (\ref{g-domain}),
with $\psi (0)=0$ and $\|\nabla \psi\|_\infty\le M$.
Consider the map $\Phi: \partial\Omega \to \R^{d-1}$, defined by $\Phi (x^\prime, \psi(x^\prime))=x^\prime$.
We say $Q\subset \partial\Omega$ is a surface cube of $\partial\Omega$ if $\Phi(Q)$ is a cube of $\R^{d-1}$ (with sides parallel to the coordinate planes).
A dilation of $Q$ is defined by $\alpha Q=\Phi^{-1} (\alpha \Phi(Q))$.
We call $z\in Q$ the center of $Q$ if $\Phi(z)$ is the center of $\Phi(Q)$.
Similarly, the side length of $Q$ is defined to be the side length of $\Phi(Q)$.

Proof of Theorems \ref{main-theorem} and \ref{main-theorem-1} is based on a real variable argument.

 \begin{thm}\label{real-theorem}
 Let $F\in L^{p_0}(2Q_0)$ for some surface cube $Q_0$ of $\partial\Omega$ and $1\le p_0<\infty$.
 Let $p_1>p_0$ and $f \in L^p (2Q_0)$ for some $p_0<p<p_1$.
 Suppose that for each surface cube $Q\subset Q_0$ with $|Q|\le \beta |Q_0|$,
 there exist two integrable functions $F_{Q}$ and $R_Q$ such that
 \begin{align}
 |F| & \le |F_Q| +|R_Q| \quad \text{ on } 2Q,\label{real-1}\\
 \left(\fint_{2Q} |R_Q|^{p_1}\, d\sigma \right)^{1/p_1}
  & \le C_1 \left\{
\left( \fint_{\alpha Q} |F|^{p_0}\, d\sigma\right)^{1/p_0}
 +\sup_{2Q_0\supset Q^\prime\supset Q}
 \left(\fint_{Q^\prime} |f|^{p_0}\, d\sigma\right)^{1/p_0}  \right\},\label{real-2}\\
\left( \fint_{2Q} |F_Q|^{p_0}\, d\sigma \right)^{1/p_0}
  &\le C_2 \sup_{2Q_0\supset Q^\prime\supset Q} \left(\fint_{Q^\prime} |f|^{p_0}\, d\sigma\right)^{1/p_0} ,\label{real-3}
 \end{align}
 where $C_1, C_2>0$ and $0<\beta<1< \alpha$. Then 
 \begin{equation}\label{real-4}
 \left(\fint_{Q_0}
 |F|^p\, d\sigma  \right)^{1/p}
 \le {C}\left(\fint_{2Q_0} |F|^{p_0}\, d\sigma\right)^{1/p_0} 
 +C \left(\fint_{2Q_0} |f|^p\, d\sigma \right)^{1/p},
 \end{equation}
 where $C>0$ depends at most on $d$, $M$, $p_0$, $p_1$, $p$,  $C_1$, $C_2$, $\alpha$ and $\beta$.
 \end{thm}

\begin{proof}
This theorem  with $p_0=1$ was formulated and proved in \cite[Theorem 3.2 and Remark 3.3]{Shen-2007}.
Its proof was inspired by a paper of Caffarelli and Peral \cite{CP-1998}.
The case $p_0>1$ follows readily from the case $p_0=1$
by considering the functions $|F|^{p_0}$ and $|f|^{p_0}$.
\end{proof}

Assume $d\ge 3$.
To prove Theorem \ref{main-theorem-1},
we fix $f\in C_0^\infty(\R^d; \mathbb{C}^m)$. By the assumption of the theorem,
there exists a weak solution  $u\in H^1_{loc}(\Omega; \mathbb{C}^m)$ 
to the elliptic system $\mathcal{L} (u)=0$ in $\Omega$ such that
$u=f$ on $\partial\Omega$ in the sense of trace and
$
\|N(u)\|_{L^{p_0}(\partial\Omega)}\le C_0 \| f\|_{L^{p_0}(\partial\Omega)},
$
where $1<p_0< \frac{2(d-1)}{d-2}$.
We need to show that $\|N (u)\|_{L^p(\partial\Omega)}\le C \| f\|_{L^p(\partial\Omega)}$
for $p_0<p< \frac{2(d-1)}{d-2} +\e$, where $\e>0$ depends only on $d$, $m$, $p_0$, $p$, $M$ and $C_0$.

To this end we fix $Q_0=Q(0, R)$, a surface cube centered at the origin with side length $R$.
Let $Q=Q(z, r)\subset Q_0$ be a  surface cube centered at $z$ with side length $r\le \beta R$, 
where $\beta\in (0,1)$ is sufficiently small.
Let $g =\varphi f$, where $\varphi$ is a smooth cut-off function such that $0\le \varphi\le 1$,
$\varphi=1$ in $\Delta_{\gamma^2  r}(z)$,  and $\varphi=0$ in $\partial\Omega\setminus \Delta_{2\gamma^2 r}(z)$,
 where  
 $$
 2Q\subset \Delta_{\gamma r}\subset \Delta_{ 2\gamma^2 r} (z)\subset 2Q_0
 $$
  and $\gamma=\gamma (M)>1$ is  large.
By the assumption there exists a weak solution $v$ to 
$\mathcal{L} (v)=0$ in $\Omega$ such that $v= \varphi f$ on $\partial\Omega$ and
\begin{equation}\label{g-1}
\| N(v)\|_{L^{p_0}(\partial\Omega)} \le C_0 \| \varphi f\|_{L^{p_0}(\partial\Omega)}.
\end{equation}
Let $w=u-v$ and define
\begin{equation}\label{g-2}
F=N(u), \quad F_Q = N(v) \quad \text{ and } \quad R_Q=  N(w).
\end{equation}
Using  $N(u)\le N(v) + N(w)$, we obtain (\ref{real-1}).
To verify (\ref{real-3}), we use the estimate (\ref{g-1}) to obtain 
\begin{equation}\label{g-3}
\aligned
\left(\fint_{2Q} |F_Q|^{p_0}\, d\sigma\right)^{1/p_0}
&\le C \left(\frac{1}{|Q|}\int_{\partial\Omega} |N(v)|^{p_0}\, d\sigma\right)^{1/p_0}\\
& \le C\left( \frac{1}{|Q|} \int_{\Delta_{2\gamma^2 r} (z)} |f|^{p_0}\, d\sigma\right)^{1/p_0} \\
& \le C \sup_{2Q_0\supset Q^\prime\supset Q} \left(\fint_{Q^\prime} |f|^{p_0}\, d\sigma\right)^{1/p_0}.
\endaligned
\end{equation}

To verify (\ref{real-2}), we  use Theorem \ref{RH-theorem}.
Observe that $\mathcal{L}(w)=0$ in $\Omega$
and $w=0$ on $\Delta_{\gamma^2 r}(z)$. By choosing $\gamma=\gamma(M)>1$ sufficiently large,
it follows from (\ref{RH-2})  as well as Remark \ref{self} that 
\begin{equation}\label{g-9}
\left(\fint_{\Delta_{\gamma r} (z)} |N^{4k\gamma r} (w)|^{\overline{q}}\, d\sigma \right)^{1/\overline{q}}
\le C
\fint_{\Delta_{2\gamma r} (z)} N^{4k\gamma r} (w)\, d\sigma, 
\end{equation}
where $\overline{q}=\frac{2(d-1)}{d-2} +\e$ and $\e>0$ depends only on 
$d$, $m$, $M$ and $C_0$.
Note that for any $y\in \Delta_{\gamma r} (z)$,
\begin{equation}\label{g-10}
N (w) (y) \le N^{4k\gamma r} (w) (y)
+ C \fint_{\Delta_{2\gamma r} (z)} N (w)\, d\sigma
\end{equation}
(see the proof of (\ref{RH-3.01})). This, together with (\ref{g-9}), yields
\begin{equation}\label{g-11}
\left(\fint_{\Delta_{\gamma r} (z)} |N (w)|^{\overline{q}}\, d\sigma \right)^{1/\overline{q}}
\le C
\fint_{\Delta_{2\gamma r} (z)} N (w)\, d\sigma.
\end{equation}
Hence,
\begin{equation}\label{g-12}
\aligned
\left(\fint_{2Q} |R_Q|^{\overline{q}} \, d\sigma \right)^{1/\overline{q}}
&\le C\left(\fint_{\Delta_{\gamma r} (z)} |N (w)|^{\overline{q}}\, d\sigma \right)^{1/\overline{q}}\\
&\le C \left(\fint_{\Delta_{2\gamma r} (z)} |N(w)|^{p_0} \, d\sigma \right)^{1/p_0}\\
&\le C \left(\fint_{\Delta_{2\gamma r}(z)} |N(u)|^{p_0} \, d\sigma \right)^{1/p_0}
+C \left(\fint_{\Delta_{2\gamma r}(z)} |N(v)|^{p_0} \, d\sigma \right)^{1/p_0}\\
&\le C \left(\fint_{\alpha Q} |F|^{p_0} \, d\sigma \right)^{1/p_0}
+ C \sup_{2Q_0\supset Q^\prime\supset Q} \left(\fint_{Q^\prime} |f|^{p_0}\, d\sigma\right)^{1/p_0},
\endaligned
\end{equation}
where $\alpha Q\supset \Delta_{2\gamma r}(z)$ and
we have used (\ref{g-1}) for the last inequality.

To summarize, we have verified the conditions in Theorem \ref{real-theorem}.
As a result, we may conclude that
\begin{equation}\label{g-13}
\left(\fint_{Q_0} |N(u)|^p\, d\sigma \right)^{1/p}
\le C \left(\fint_{2Q_0} |N(u)|^{p_0}\, d\sigma\right)^{1/p_0}
+ C \left(\fint_{2Q_0} |f|^p\, d\sigma \right)^{1/p}
\end{equation}
for any $p_0<p< \frac{2(d-1)}{d-2} +\e$.
It follows that
$$
\aligned
\left(\int_{Q_0} |N(u)|^p\, d\sigma \right)^{1/p}
 & \le C |Q_0|^{(d-1)(\frac{1}{p}-\frac{1}{p_0})}
\left(\int_{2Q_0} |N(u)|^{p_0}\, d\sigma\right)^{1/p_0}
+ C \left(\int_{2Q_0} |f|^p\, d\sigma \right)^{1/p}\\
&\le C |Q_0|^{(d-1)(\frac{1}{p}-\frac{1}{p_0})} \| f\|_{L^{p_0} (\partial\Omega)}
+ C \| f\|_{L^p(\partial\Omega)}.
\endaligned
$$
By letting the side length of $Q_0$ go to infinity in the inequalities above, we obtain 
the desired estimate 
$\|N(u)\|_{L^p(\partial\Omega)} \le C \| f\|_{L^p(\partial\Omega)}$. 

Finally, note that if $d=2$, the same argument yields  the estimate (\ref{estimate-1.0})
for $p_0<p<\infty$.



\section{\bf Proof of Theorem \ref{main-theorem}}

Theorem \ref{main-theorem} follows from the proof of Theorem \ref{main-theorem-1}
by a simple localization technique.
Fix $z\in \partial\Omega$.
Let $r_0=\text{diam}(\Omega)$ and $r=c_0 r_0$, where $c_0>0$ is sufficiently small such that 
$$
B(z, r)\cap\Omega =B(z, r) \cap \big\{ (x^\prime, x_d):\,  x_d>\psi (x^\prime) \big\}
$$
in a new coordinate system, obtained from the standard system through translation and rotation.
It follows from the estimate (\ref{g-13}) that 
\begin{equation}\label{b-1}
\aligned
\left(\int_{B(z, c_1 r)\cap \partial\Omega} |N(u)|^p\, d\sigma\right)^{1/p}
 & \le C r_0^{(d-1)(\frac{1}{p}-\frac{1}{p_0})} 
\|N(u)\|_{L^{p_0} (\partial\Omega)} 
+ C \| f\|_{L^p(\partial\Omega)}\\
& 
\le C r_0^{(d-1)(\frac{1}{p}-\frac{1}{p_0})} 
\|f \|_{L^{p_0} (\partial\Omega)} 
+ C \| f\|_{L^p(\partial\Omega)}\\
& \le C \| f\|_{L^p(\partial\Omega)},
\endaligned
\end{equation}
where $c_1=c_1(\Omega)>0$ is small
and we have used H\"older's inequality as well as the fact $|\partial\Omega|\le C r_0^{d-1}$ for the last step.
By covering $\partial\Omega$ with a finite number of balls $\{B(z_\ell, c_1 r)\}$ we obtain the estimate
(\ref{estimate-1.0}).

 \bibliographystyle{amsplain}
 
\bibliography{extra-2017.bbl}

\bigskip

\begin{flushleft}
Zhongwei Shen,
Department of Mathematics,
University of Kentucky,
Lexington, Kentucky 40506,
USA.

E-mail: zshen2@uky.edu
\end{flushleft}

\bigskip

\medskip

\end{document}